\documentclass{amsart}
\usepackage{amsmath}
\usepackage{multicol}
\usepackage{subfigure}
\usepackage{amssymb}
\usepackage{enumitem}
\usepackage{amsthm}
\usepackage{hyperref}

\subjclass[2010]{32A08}
\keywords{ LINEAR FRACTIONAL MAPS, CROSS RATIO, TRANSITIVITY}

\theoremstyle{theorem}
\newtheorem{theorem}{Theorem}

\theoremstyle{lemma}
\newtheorem{lemma}{Lemma}
 
\theoremstyle{definition}
\newtheorem*{definition}{Definition}
\newtheorem*{example}{Example}
\newtheorem*{question}{Question}

\begin{document}

\title{A Generalized Cross Ratio}
\markright{A Generalized Cross Ratio}
\author{Michael R. Pilla \\ \href{mailto:mpilla@iu.edu}{mpilla@iu.edu} \\ Swain Hall East 205 \\ 729 East 3rd St. \\ Bloomington, IN 47405}

\maketitle

\begin{abstract}
In one complex variable, the cross ratio is a well-known quantity associated with four given points in the complex plane that remains invariant under linear fractional maps. In particular, if one knows where three points in the complex plane are mapped under a linear fractional map, one can use this invariance to explicitly determine the map and to show that linear fractional maps are $3$-transitive. In this paper, we define a generalized cross ratio and determine some of its basic properties. In particular, we determine which hypotheses must be made to guarantee that our generalized cross ratio is well defined. We thus obtain a class of maps that obey similar transitivity properties as in one complex dimension, under some more restrictive conditions.
\end{abstract}

\section{Background.}

In 1872, Felix Klein famously introduced a point of view regarding what geometry should be about \cite{Klein}. Known as the Erlangen program, Klein viewed geometry as a study of invariants under group transformations. An important and robust example of such an invariant quantity is given by the cross ratio.

Given four finite distinct points $z_1$, $z_2$, $z_3$, and $z_4$ in the complex plane, the cross ratio is defined as 

\begin{equation}\label{lfm1}
(z_1, z_2 , z_3, z_4)=\frac{(z_1-z_3)(z_2-z_4)}{(z_1-z_4)(z_2-z_3)}.
\end{equation}

If $z_i=\infty$ for some $i=1,2,3,4$ then we cancel the terms with $z_i$. For example, if $z_1=\infty$, then 

$$(z_1, z_2, z_3, z_4)=\frac{(z_2-z_4)}{(z_2-z_3)}.$$

With this definition, we may view the cross ratio as a function of $z$ given by $(z, z_1 , z_2, z_3)$. Recall that a linear fractional map, also known as a Mobius transformation, is defined as 

$$f(z)=\frac{az+b}{cz+d}$$

\noindent where the coefficients $a$, $b$, $c$, and $d$ are complex numbers such that $ad-bc \neq 0$ (otherwise $f(z)$ is a constant function).

It is well known that the cross ratio is preserved under linear fractional maps. Thus if $\phi$ is a linear fractional map in one complex variable and $\phi(z_i)=w_i$ for $i=1,2,3$ then 

$$\frac{(w-w_2)(w_1-w_3)}{(w-w_3)(w_1-w_2)}=\frac{(z-z_2)(z_1-z_3)}{(z-z_3)(z_1-z_2)}$$

\noindent where $w=\phi(z)$. In such a case one may solve for $w$ to determine the map explicitly.

\section{Homogeneous coordinates.}

Given a linear fractional map

$$\phi(z)=\frac{az+b}{cz+d}$$

\noindent we may define the associated matrix of $\phi$ as 

\begin{equation*}
    m_{\phi} = \begin{pmatrix}
        a & b \\
        c & d
     \end{pmatrix}.
\end{equation*}

Although linear fractional maps can be associated with matrices, they are clearly not linear transformations of $\mathbb{C}$ (take $f(z)=\frac{1}{z}$, for example). Instead, the associated matrix acts as a linear transformation on what are known as homogeneous coordinates in $\mathbb{C}^2$. Homogeneous coordinates are defined as follows. We associate the point $u=(u_1, u_2) \in \mathbb{C}^2$ \textbackslash $\{0\}$ with the point $z=\frac{u_1}{u_2} \in \mathbb{C}$. We will write $z \sim u$ if $z \in \mathbb{C}$ is the point associated with $u \in \mathbb{C}^2$. This space is equivalent to the Riemann sphere and is known as the complex projective line  ${\bf CP}^1$. The point $(1,0)$ is associated with the point at infinity on the Riemann sphere.  Notice if $\phi(z)=w$ and the point $u \in {\bf CP}^1$ is associated with $z$, then $v=m_{\phi}u$ is associated with the point $w$ and vice versa. That is, if $z \sim u$, then $\phi(z) \sim m_{\phi}u$.

The reason we are introducing these homogeneous coordinates is made evident by the following discussion. A linear transformation in $\mathbb{C}^2$ can be represented by a complex matrix as

$$\begin{pmatrix}v_1  \\  v_2\end{pmatrix}=\begin{pmatrix}a & b  \\  c & d \end{pmatrix}\begin{pmatrix}u_1  \\  u_2\end{pmatrix}=\begin{pmatrix}au_1+bu_2  \\  cu_1 +du_2 \end{pmatrix}.$$

Let $z \sim \begin{pmatrix}u_1  \\  u_2\end{pmatrix}$ and $w \sim \begin{pmatrix}v_1  \\  v_2\end{pmatrix}$. Then we can associate the above linear transformation with the linear fractional transformation 

$$w=\frac{v_1}{v_2}=\frac{au_1+bu_2}{cu_1+du_2}=\frac{a\left(\frac{u_1}{u_2}\right)+b}{c\left(\frac{u_1}{u_2}\right)+d}=\frac{az+b}{cz+d}.$$

Thus, the associated matrices are linear transformations acting on ${\bf CP}^1$! (see [$4$] for more details) A routine calculation shows that $m_{\phi_1 \circ \phi_2}=m_{\phi_1}m_{\phi_2}$ and $m_{\phi^{-1}}=(m_{\phi})^{-1}$ as well.

We next define the cross ratio in homogeneous coordinates. As we will see, this is the key to generalized the definition to higher dimensions. First we define the following quantity. If $u=(u_1, u_2)$ and $v=(v_1, v_2)$ are points in ${\bf CP}^1$, then we define

\begin{equation}
\label{eqn:def}
[u,v]=\det \begin{pmatrix}
u_1 & v_1\\ 
u_2 & v_2 
\end{pmatrix}.
\end{equation}

Since we have reserved subscripts for the components of a point in several variables, for the remainder of our discussion we choose to designate {\it superscripts} for distinct points. As we will only consider linear fractional maps, there should be no confusion about notation. Just remember, superscripts will not designate exponents but rather distinct points. We may then define the cross ratio as follows.

\begin{definition}
Let $u^1$, $u^2$, $u^3$, and $u^4$ be four distinct points in  ${\bf CP}^1$, then we define the cross ratio to be

$$(u^1, u^2, u^3, u^4)=\frac{[u^1,u^3][u^2,u^4]}{[u^1,u^4][u^2,u^3]}.$$

\end{definition}

Amazingly, this definition agrees with the traditional definition of the cross ratio. For finite points we can see this as follows. Given four finite distinct points $z^1$, $z^2$, $z^3$, and $z^4$ in $\mathbb{C}$, take a moment to convince yourself that we may write the associated points as $u^i=(z^i, 1)$ for $i=1,2,3,4$ and find 

\begin{align*}(z^1, z^2, z^3, z^4)&=\frac{(z^1-z^3)(z^2-z^4)}{(z^1-z^4)(z^2-z^3)}=\frac{\det \begin{pmatrix}
z^1 & z^3\\ 
1 & 1 
\end{pmatrix}\det \begin{pmatrix}
z^2 & z^4\\ 
1 & 1 
\end{pmatrix}}{\det \begin{pmatrix}
z^1 & z^4\\ 
1 & 1 
\end{pmatrix}\det \begin{pmatrix}
z^2 & z^3 \\ 
1 & 1 
\end{pmatrix}}\\
&=\frac{[u^1,u^3][u^2,u^4]}{[u^1,u^4][u^2,u^3]}=(u^1,u^2,u^3,u^4).
\end{align*}

If one of our points is associated with the point at infinity, then without loss of generality, we let $z^1=\infty$ and the point associated with $z^1$ may, without loss of generality, be written as $u^1=(1,0)$ which gives us

\begin{align*}(z^1, z^2, z^3, z^4)&=\frac{(z^2-z^4)}{(z^2-z^3)}=\frac{\det \begin{pmatrix}
z^2 & z^4\\ 
1 & 1 
\end{pmatrix}}{\det \begin{pmatrix}
z^2 & z^3 \\ 
1 & 1 
\end{pmatrix}}\\
&=\frac{\det \begin{pmatrix}
1 & z^3\\ 
0 & 1 
\end{pmatrix}\det \begin{pmatrix}
z^2 & z^4\\ 
1 & 1 
\end{pmatrix}}{\det \begin{pmatrix}
1 & z^4\\ 
0 & 1 
\end{pmatrix}\det \begin{pmatrix}
z^2 & z^3 \\ 
1 & 1 
\end{pmatrix}}=\frac{[u^1,u^3][u^2,u^4]}{[u^1,u^4][u^2,u^3]}=(u^1, u^2,u^3,u^4).
\end{align*}

Thus, the cross ratio is not only invariant under $\phi$  but also $m_{\phi}$. In addition to this, the definition of the cross ratio in terms of homogeneous coordinates is more unifying in the sense that the case of ``infinity" is included in the definition. This resonates with Klein's Erlangen program and its focus on projective geometry (geometry done on projective spaces) as a more unifying geometry. In particular, we utilize this point of view to define a cross ratio in several complex variables.

\section{Linear Fractional Maps in Several Complex Variables.}

To define the cross ratio in several complex variables, we must first say what it means to be a linear fractional map in $\mathbb{C}^N$ for $N>1$. We saw that in one variable the linear fractional maps were not linear transformations on $\mathbb{C}$ but rather on the complex projective line ${\bf CP}^1$. In this paper, we take the perspective that we would like our linear fractional maps in $\mathbb{C}^N$ to be linear transformations on the complex projective space ${\bf CP}^N$.

Thus we associate the point $z=(z_1', z_2)$  where $z_1' \in \mathbb{C}^N$ and $z_2 \in \mathbb{C}$, $z \neq 0$, with the point $\frac{z_1'}{z_2} \in \mathbb{C}^N$. This associated space is known as the complex projective space ${\bf CP}^N$. We now consider a linear transformation in ${\bf CP}^N$ which can be represented by a complex matrix as

\begin{equation*}
  \begin{pmatrix}
        A & B \\
        C^* & D
     \end{pmatrix}
\end{equation*}

\noindent where $A$ is an $N \times N$ matrix, $B$ and $C$ are column vectors in $\mathbb{C}^N$, $D \in \mathbb{C}$, and $C^*$ represents the conjugate transpose of $C$. Denote the rows of $A$ by $a_i$ for $i=1,...,N$ and  $B=\begin{pmatrix}b_1 & \cdots & b_N \end{pmatrix}^T$ and let $\langle \cdot, \cdot \rangle$ be the standard inner product. For a point $\begin{pmatrix}z_1'  \\  z_2\end{pmatrix}$ in ${\bf CP}^N$, we have

$$\begin{pmatrix}w_1'  \\  w_2\end{pmatrix}=\begin{pmatrix}A & B  \\  C^* & D \end{pmatrix}\begin{pmatrix}z_1'  \\  z_2\end{pmatrix}=\begin{pmatrix}\langle a_1, \overline{z_1'}\rangle+b_1z_2  \\ \vdots \\ \langle a_N, \overline{z_1'}\rangle+b_Nz_2 \\ \langle z_1',C \rangle+Dz_2 \end{pmatrix}.$$

Let $z \sim \begin{pmatrix}z_1'  \\  z_2\end{pmatrix}$ and $w \sim \begin{pmatrix}w_1'  \\  w_2\end{pmatrix}$. Then we can associate the above linear transformation in ${\bf CP}^N$ with the not necessarily linear transformation in $\mathbb{C}^N$ given by

\begin{align*}w=\frac{w_1'}{w_2}&=\left(\frac{\langle a_1, \overline{z_1'}\rangle+b_1z_2 }{\langle z_1',C \rangle+Dz_2},...,\frac{\langle a_N, \overline{z_1'}\rangle+b_Nz_2 }{\langle z_1',C\rangle+Dz_2}\right)\\
&=\left(\frac{\langle a_1, \overline{\frac{z_1'}{z_2}}\rangle+b_1 }{\langle\frac{z_1'}{z_2},C\rangle+D},...,\frac{\langle a_N, \overline{\frac{z_1'}{z_2}}\rangle+b_N }{\langle \frac{z_1'}{z_2},C\rangle+D}\right)\\
&=\left(\frac{\langle a_1, \overline{z}\rangle+b_1 }{\langle z,C\rangle+D},...,\frac{\langle a_N, \overline{z}\rangle+b_N }{\langle z,C\rangle+D}\right)\\
&=\frac{Az+B}{\langle z,C\rangle+D}.
\end{align*}

This is precisely the definition given by Cowen and MacCluer \cite{Cowen} and is the one we will adopt.

\begin{definition}
We say $\phi$ is a linear fractional map in $\mathbb{C}^N$ if 

\begin{equation}
\label{LFM}
\phi(z)=\frac{Az+B}{\langle z,C\rangle+D}.
\end{equation}

\noindent where $A$ is an $N \times N$ matrix, $B$ and $C$ are column vectors in $\mathbb{C}^N$, $D \in \mathbb{C}$, and $\langle \cdot, \cdot \rangle$ is the standard inner product.
\end{definition}

This class of maps has been studied in more generality by others \cite{H}, \cite{P}, \cite{S}, \cite{Sm}.

We define the associated matrix $m_{\phi}$ of the linear fractional map $\phi(z)=\frac{Az+B}{\langle z,C\rangle+D}$ to be given by

\begin{equation*}
    m_{\phi} = \begin{pmatrix}
        A & B \\
        C^* & D
     \end{pmatrix}
\end{equation*}

\noindent which, as we saw, is a linear transformation on ${\bf CP}^N$. If $\phi(z)=w$ and the point $u \in \mathbb{C}^{N+1}$ is associated with $z$, then $v=m_{\phi}u$ is associated with the point $w$ and vice versa. A routine calculation also shows that $m_{\phi_1 \circ \phi_2}=m_{\phi_1}m_{\phi_2}$ and $m_{\phi^{-1}}=(m_{\phi})^{-1}$. Thus, as we did in one variable, we can toggle back and forth between the complex space $\mathbb{C}^N$ and ${\bf CP}^N$ at our convenience.

\begin{example}
Let $\phi$ be the linear fractional map in two complex variables given by

$$\phi(z)=\phi(z_1,z_2)=\left( \frac{z_1+1}{-z_1+3}, \frac{2z_2}{-z_1+3} \right).$$

Identifying $\langle z, C \rangle$ with $C^*z$, we can write this as 

$$\left( \frac{z_1+1}{-z_1+3}, \frac{2z_2}{-z_1+3} \right)=\frac{\begin{pmatrix}
       1 & 0  \\
       0 & 2 
     \end{pmatrix}\begin{pmatrix}
       z_1 \\
       z_2 
     \end{pmatrix}
+\begin{pmatrix}
       1 \\
       0 
     \end{pmatrix}}{(-1, 0)^T(z_1, z_2)+3}.$$

Thus letting $A=\begin{pmatrix}
       1 & 0  \\
       0 & 2 
     \end{pmatrix}$, $B=(1, 0)^T$, $C=(-1, 0)^T$, and $D=3$, we see that the associated matrix of $\phi$ is then given by
\begin{equation*}
    m_{\phi} = \begin{pmatrix}
        A & B \\
        C^* & D
     \end{pmatrix}=
 \begin{pmatrix}
        \phantom{-}1 & 0 & 1 \\
       \phantom{-} 0 & 2 & 0 \\
       -1 & 0 & 3
     \end{pmatrix}.
\end{equation*}

\end{example}

\section{The Cross Ratio in Several Complex Variables.}

While generalized cross ratios have been defined previously in terms of the Schwarzian derivative \cite{Gong}, this paper takes a different perspective that is more tractable and attempts to be in the elementary spirit of the traditional cross ratio. For one complex variable, we saw that the cross ratio is invariant under the maps $\phi$ and $m_{\phi}$. We would like the same to be true for the cross ratio in $\mathbb{C}^N$. The cross ratio can also be utilized to show that any three pairwise distinct points in $\mathbb{C}$ that are sent to another set of three pairwise distinct points uniquely determine a linear fractional map in $\mathbb{C}$. The proof of this uses the fact that one can utilize cross ratios to define a unique map that sends the pairwise distinct triple $(z^1,z^2,z^3)$  to the standardized points $(1,0, \infty)$. Generalizing to higher dimensions, however, does not come without cost. Even for $N=2$, one may recall that projective transformations map lines to lines.  Without further hypotheses to avoid problems of collinearity, the task of mapping a pairwise distinct quadruple $(z^1,z^2,z^3, z^4)$ to four standardized points is hopeless. Treading carefully, we use these facts to motivate our definition of a cross ratio in $\mathbb{C}^N$ for $N>1$. In particular we seek a minimal definition that will preserve the properties of the cross ratio just described.

First, recalling that subscripts represent coordinates (components of a vector) and superscripts distinguish distinct points in higher dimensions, we define the following quantity. If we let $u^i=(u^i_1, u^i_2,...,u^i_{N+1})$ for $i=1,...,N+1$ be elements of ${\bf CP}^N$, then we define

$$[u^1, u^2,...,u^{N+1}]=\det \begin{pmatrix}
    u^1_1 & u_1^2 & u_1^3 & \dots  & u_1^{N+1} \\
    u_2^1 & u_2^2 & u_2^3 & \dots  & u_2^{N+1} \\
    \vdots & \vdots & \vdots & \ddots & \vdots \\
    u_{N+1}^1 & u_{N+1}^2 & u_{N+1}^3 & \dots  & u_{N+1}^{N+1}
\end{pmatrix}.$$

This is a natural generalization of the definition given by Equation \ref{eqn:def}.

In order to motivate our general definition, we began by defining a cross ratio in $\mathbb{C}^2$ and proceed to generalize to $\mathbb{C}^N$.

\begin{definition}
Given five distinct points $u^i$ for $i=1,...,5$ in ${\bf CP}^2$, we define the cross ratio as 

$$(u^1, u^2, u^3, u^4, u^5)=\frac{[u^1, u^3, u^5][u^2, u^4, u^5]}{[u^1, u^4, u^5][u^2, u^3, u^5]}.$$

\end{definition}

One motivation for choosing this particular cross ratio is because it ``reduces" to our usual cross ratio using the appropriate coordinates. We will also find it useful when addressing questions of transitivity. Let's look at an example of how, under the right conditions, this definition ``reduces" to the standard definition of the cross ratio. For the vectors $u^i=(u^i_1, u^i_2, u^i_3)$ with $i=1,..,4$, if we standardize coordinates so that $u^i_3=1$, $z^i_1=u^i_1/u^i_3$, and $z^i_2= u^i_2/u^i_3$, we get vectors in $\mathbb{C}^2/\{0\}$ which we can interpret as ${\bf CP}^1$. We still have one extra vector $u^5$ though. However, since we let $u^i_3=1$ for $i=1,...,4$, why not choose $u^5$ to be a ``point at infinity". In our context, this means choosing the final coordinate to equal $0$. If we choose $u^5=(0,1,0)$, which is {\it a} point on the line at infinity, we have

\begin{align*}
&(u^1, u^2, u^3, u^4, u^5)=\frac{[u^1, u^3, u^5][u^2, u^4, u^5]}{[u^1, u^4, u^5][u^2, u^3, u^5]}\\
&=\frac{ \det\begin{pmatrix}
    u^1_1 & u^3_1 & 0 \\
    u^1_2 & u^3_2 & 1 \\
    u^1_3 & u^3_3 & 0
\end{pmatrix}\det\begin{pmatrix}
    u^2_1 & u^4_1 & 0 \\
    u^2_2 & u^4_2 & 1 \\
    u^2_3 & u^4_3 & 0
\end{pmatrix}}
{\det\begin{pmatrix}
    u^1_1 & u^4_1 & 0 \\
    u^1_2 & u^4_2 & 1 \\
    u^1_3 & u^4_3 & 0
\end{pmatrix}\det\begin{pmatrix}
    u^2_1 & u^3_1 & 0 \\
    u^2_2 & u^3_2 & 1 \\
    u^2_3 & u^3_3 & 0
\end{pmatrix}}
=\frac{ \det\begin{pmatrix}
    z^1_1 & z^3_1 & 0 \\
    z^1_2 & z^3_2 & 1 \\
    1 & 1 & 0
\end{pmatrix}\det\begin{pmatrix}
   z^2_1 & z^4_1 & 0 \\
    z^2_2 & z^4_2 & 1 \\
    1 & 1 & 0
\end{pmatrix}}
{\det\begin{pmatrix}
   z^1_1 & z^4_1 & 0 \\
    z^1_2 & z^4_2 & 1 \\
    1 & 1 & 0
\end{pmatrix}\det\begin{pmatrix}
    z^2_1 & z^3_1 & 0 \\
    z^2_2 & z^3_2 & 1 \\
    1 & 1 & 0
\end{pmatrix}}\\
&=\frac{(z^1-z^3)(z^2-z^4)}{(z^1-z^4)(z^2-z^3)}=(z^1, z^2 , z^3, z^4).
\end{align*}

In $\mathbb{C}$, there are $4!=24$ permutations of $(z^1, z^2, z^3, z^4)$ of which $6$ are distinct. In $\mathbb{C}^2$, we recognize that 

$$(z^i, z^j, z^k , z^l, z^m )=(z^k, z^l, z^i , z^j, z^m )=(z^j, z^i, z^l , z^k, z^m)=(z^l, z^k, z^j , z^i, z^m)$$

\noindent and that no other permutation is equal to $(z^i, z^j, z^k , z^l, z^m )$. Thus, we have five ways to choose $m$ and once $m$ is chosen, we start with $i$ (by the above equation we may, equivalently, start with any of $i,j,k,l$ and have $3!=6$ ways to choose the remaining three values which gives us $5 \times 6=30$ distinct permutations. 

Let's see, explicitly, what this definition looks in like our usual complex variables when the $z^i$'s for $i=1,...,5$ are finite complex numbers where we are once again putting the notation for the distinct points in the superscript, leaving the subscripts for vector components. This means we set the third coordinate to $1$ in each case. Thus we have

\begin{align*}
&(u^1, u^2, u^3, u^4, u^5)=\frac{[u^1, u^3, u^5][u^2, u^4, u^5]}{[u^1, u^4, u^5][u^2, u^3, u^5]}\\
&=\frac{ \det\begin{pmatrix}
    u^1_1 & u^3_1 & u^5_1 \\
    u^1_2 & u^3_2 & u^5_2 \\
    u^1_3 & u^3_3 & u^5_3
\end{pmatrix}\det\begin{pmatrix}
    u^2_1 & u^4_1 & u^5_1 \\
    u^2_2 & u^4_2 & u^5_2 \\
    u^2_3 & u^4_3 & u^5_3
\end{pmatrix}}
{\det\begin{pmatrix}
    u^1_1 & u^4_1 & u^5_1 \\
    u^1_2 & u^4_2 & u^5_2 \\
    u^1_3 & u^4_3 & u^5_3
\end{pmatrix}\det\begin{pmatrix}
    u^2_1 & u^3_1 & u^5_1 \\
    u^2_2 & u^3_2 & u^5_2 \\
    u^2_3 & u^3_3 & u^5_3
\end{pmatrix}}
=\frac{ \det\begin{pmatrix}
    z^1_1 & z^3_1 & z^5_1 \\
    z^1_2 & z^3_2 & z^5_2 \\
    1 & 1 & 1
\end{pmatrix}\det\begin{pmatrix}
    z^2_1 & z^4_1 & z^5_1 \\
    z^2_2 & z^4_2 & z^5_2 \\
    1 & 1 & 1
\end{pmatrix}}
{\det\begin{pmatrix}
    z^1_1 & z^4_1 & z^5_1 \\
    z^1_2 & z^4_2 & z^5_2 \\
    1 & 1 & 1
\end{pmatrix}\det\begin{pmatrix}
    z^2_1 & z^3_1 & z^5_1 \\
    z^2_2 & z^3_2 & z^5_2 \\
    1 & 1 & 1
\end{pmatrix}}\\
&=\frac{\left(z^1_1(z^3_2-z^5_2)-z^3_1(z^1_2-z^5_2)+z^5_1(z^1_2-z^3_2)\right)\left(z^2_1(z^4_2-z^5_2)-z^4_1(z^2_2-z^5_2)+z^5_1(z^2_2-z^4_2)\right)}{\left(z^1_1(z^4_2-z^5_2)-z^4_1(z^1_2-z^5_2)+z^5_1(z^1_2-z^4_2)\right)\left(z^2_1(z^3_2-z^5_2)-z^3_1(z^2_2-z^5_2)+z^5_1(z^2_2-z^3_2)\right)}\\&=(z^1, z^2 , z^3, z^4,z^5).
\end{align*}

This expression is definitely cleaner in homogeneous coordinates! As an exercsie for the reader, try writing down the expression when one of the $z^i$'s is a point on the line at infinity.

Next we recall that by letting $z=z^1$ be a complex variable, the cross ratio in one complex variable becomes the linear fractional map given by

$$(z, z^2, z^3, z^4)=\frac{(z-z^3)(z^2-z^4)}{(z-z^4)(z^2-z^3)}.\footnote{This is just Equation \ref{lfm1} with superscript notation instead.}$$

As we saw earlier, this definition coincides with the definition given in terms of homogeneous coordinates. Notice this cross ratio sends the points $z^2$, $z^3$, and $z^4$ to $1$, $0$, and $\infty$, respectively. We would like our generalized cross ratio to be able to recover this highlight in some way. In order to achieve this, we first have to extend our definition. After all, if we let $z=z^1=(z_1, z_2)$, then the above expression for $(z^1, z^2 , z^3, z^4,z^5)$ does not give us a linear fractional map as given in Definition \ref{LFM}. Such a map should map $\mathbb{C}^2$ to $\mathbb{C}^2$, not $\mathbb{C}^2$ to $\mathbb{C}$! Thus we define the cross ratio pair.
\begin{definition}

We define the cross ratio pair in ${\bf CP}^2$ as

$$(u^1, u^2, u^3, u^4, u^5)_2=\left( \frac{[u^1, u^3, u^5][u^2, u^4, u^5]}{[u^1, u^4, u^5][u^2, u^3, u^5]}, \frac{[u^1, u^3, u^4][u^2, u^4, u^5]}{[u^1, u^4, u^5][u^2, u^3, u^4]} \right)$$

\noindent where we see that this defines a linear fractional map when the point associated with $u^1$ is a variable in $\mathbb{C}^2$. 

\end{definition}

Motivated by the above results, we define a generalized cross ratio. Although the notation is a bit more intimidating, this definition is nothing more than the natural extension of the two variable case.

\begin{definition}
Given $N+3$ distinct points $u^i$ for $i=1,...,N+3$ in ${\bf CP}^N$, we define the cross ratio as 

$$(u^1, u^2, ...,u^{N+2}, u^{N+3})=\frac{[u^1, u^3, u^5,..., u^{N+3}][u^2, u^4, u^5,..., u^{N+3}]}{[u^1, u^4, u^5,..., u^{N+3}][u^2, u^3, u^5,..., u^{N+3}]}$$

\noindent where each ellipsis represents the ordered sequence of omitted digits. 
\end{definition}

To simplify the definition of the cross ratio $N$-tuple, we introduce new notation. We define $[u^i, u^j]_N^c$ by

\begin{equation*}[u^i, u^j]_N^c=[u^1,...,u^{i-1},u^{i+1},..., u^{j-1}, u^{j+1},..., u^{N+3}].
\end{equation*}

This sets the stage for our general definition. If the notation appears intimidating, just reminder yourself that it is the natural generalization of the definition of the cross ratio pair in ${\bf CP}^2$.

\begin{definition} 

We define the cross ratio $N$-tuple to be given by

$$(u^1, u^2,..., u^{N+2}, u^{N+3})_N=\left( \bigg\{\frac{[u^2, u^i]_N^c[u^1, u^3]_N^c}{[u^1, u^i]_N^c[u^2, u^3]_N^c}\bigg\}_{i=4,...,N+3}\right)$$

\noindent where the curly brackets denote a sequence over $i=4,...,N+3$.

\end{definition}

As in the two variable case, this defines a linear fractional map when the point associated with $u^1$ is a variable in $\mathbb{C}^N$.

We saw that we could ``reduce" our cross ratio in two variables to the one variable definition. The next theorem tells us that this works in any dimension. The reasoning, again, is a natural extension of what was demonstrated in two variables.

\begin{theorem}
Identifying coordinates $z^i_j=\frac{u^i_j}{u^{N+1}_j}$ for $j=1,...,N$ and $u^{N+3}=(0,0,...,0,1,0)^T$, the cross ratio in ${\bf CP}^N$ reduces to the cross ratio in ${\bf CP}^{N-1}$.
\end{theorem}

\begin{proof}
We use proof by induction. We have seen that our result is true for $N=2$. Now suppose $N=k+1$ where $k>1$, making the appropriate identifications, we have
\begin{align*}
&(u^1, u^2, ..., u^{(k+1)+3})=\frac{[u^1, u^3, u^5,..., u^{k+4}][u^2, u^4, u^5,..., u^{k+4}]}{[u^1, u^4, u^5,..., u^{k+4}][u^2, u^3, u^5,..., u^{k+4}]}\\
&=\frac{ \det \begin{pmatrix}
    u^1_1 & u^3_1 & u_1^5 & \dots  & u^{k+4}_1 \\
   u^1_2 & u^3_2 & u_2^5 & \dots  & u^{k+4}_2 \\
    \vdots & \vdots & \vdots & \ddots & \vdots \\
    u^1_{k+2} & u^3_{k+2} & u_{k+2}^5 & \dots  & u^{k+4}_{k+2}
\end{pmatrix}\det \begin{pmatrix}
    u_1^2 & u_1^4 & u_1^5 & \dots  & u_1^{k+4} \\
    u_2^2 & u_2^4 & u_2^5 & \dots  & u_2^{k+4} \\
    \vdots & \vdots & \vdots & \ddots & \vdots \\
    u_{k+2}^2 & u_{k+2}^4 & u_{k+2}^5 & \dots  & u_{k+2}^{k+4}
\end{pmatrix}}
{\det \begin{pmatrix}
    u_1^1 & u_1^4 & u_1^5 & \dots  & u_1^{k+4} \\
    u_2^1 & u_2^4 & u_2^5 & \dots  & u_2^{k+4} \\
    \vdots & \vdots & \vdots & \ddots & \vdots \\
    u_{k+2}^1 & u_{k+2}^4 & u_{k+2}^5 & \dots  & u_{k+2}^{k+4}
\end{pmatrix}\det \begin{pmatrix}
    u_1^2 & u_1^3 & u_1^5 & \dots  & u_1^{k+4} \\
    u_2^2 & u_2^3 & u_2^5 & \dots  & u_2^{k+4} \\
    \vdots & \vdots & \vdots & \ddots & \vdots \\
    u_{k+2}^2 & u_{k+2}^3 & u_{k+2}^5 & \dots  & u_{k+2}^{k+4}
\end{pmatrix}}\\
&=\frac{ \det \begin{pmatrix}
    z_1^1 & z_1^3 & z_1^5 & \dots  & 0 \\
    z_2^1 & z_2^3 & z_2^5 & \dots  & 0 \\
    \vdots & \vdots & \vdots & \ddots & \vdots \\
    z_{k+1}^1 & z_{k+1}^3 & z_{k+1}^5 & \dots  & 1 \\
    1 & 1 & 1& \dots  & 0
\end{pmatrix}\det \begin{pmatrix}
    z_1^2 & z_1^4 & z_1^5 & \dots  & 0 \\
    z_2^2 & z_2^4 & z_2^5 & \dots  & 0 \\
    \vdots & \vdots & \vdots & \ddots & \vdots \\
    z_{k+1}^2 & z_{k+1}^4 & z_{k+1}^5 & \dots  & 1 \\
    1 & 1 & 1& \dots  & 0
\end{pmatrix}}
{\det \begin{pmatrix}
    z_1^1 & z_1^4 & z_1^5 & \dots  & 0 \\
    z_2^1 & z_2^4 & z_2^5 & \dots  & 0 \\
    \vdots & \vdots & \vdots & \ddots & \vdots \\
    z_{k+1}^1 & z_{k+1}^4 & z_{k+1}^5 & \dots  & 1 \\
    1 & 1 & 1& \dots  & 0
\end{pmatrix}\det \begin{pmatrix}
    z_1^2 & z_1^3 & z_1^5 & \dots  & 0 \\
    z_2^2 & z_2^3 & z_2^5 & \dots  & 0 \\
    \vdots & \vdots & \vdots & \ddots & \vdots \\
    z_{k+1}^2 & z_{k+1}^3 & z_{k+1}^5 & \dots  & 1 \\
    1 & 1 & 1& \dots  & 0
\end{pmatrix}}\\ 
&=\frac{ \det \begin{pmatrix}
    z_1^1 & z_1^3 & z_1^5 & \dots  &  z_1^{k+3} \\
    z_2^1 & z_2^3 & z_2^5 & \dots  & z_2^{k+3}  \\
    \vdots & \vdots & \vdots & \ddots & \vdots \\
    z_{k}^1 & z_{k}^3 & z_{k}^5 & \dots  & z_{k}^{k+3}  \\
    1 & 1 & 1 & \dots & 1
\end{pmatrix}\det \begin{pmatrix}
    z_1^2 & z_1^4 & z_1^5 & \dots  & z_1^{k+3} \\
    z_2^2 & z_2^4 & z_2^5 & \dots  & z_2^{k+3} \\
    \vdots & \vdots & \vdots & \ddots & \vdots \\
    z_{k}^2 & z_{k}^4 & z_{k}^5 & \dots  &  z_{k}^{k+3} \\
    1 & 1 & 1 & \dots & 1
\end{pmatrix}}
{\det \begin{pmatrix}
    z_1^1 & z_1^4 & z_1^5 & \dots  & z_1^{k+3} \\
    z_2^1 & z_2^4 & z_2^5 & \dots  & z_2^{k+3} \\
    \vdots & \vdots & \vdots & \ddots & \vdots \\
    z_{k}^1 & z_{k}^4 & z_{k}^5 & \dots  &  z_{k}^{k+3} \\
    1 & 1 & 1 & \dots & 1
\end{pmatrix}\det \begin{pmatrix}
    z_1^2 & z_1^3 & z_1^5 & \dots  & z_1^{k+3} \\
    z_2^2 & z_2^3 & z_2^5 & \dots  & z_2^{k+3} \\
    \vdots & \vdots & \vdots & \ddots & \vdots \\
    z_{k}^2 & z_{k}^3 & z_{k}^5 & \dots  &  z_{k}^{k+3} \\
    1 & 1 & 1 & \dots & 1
\end{pmatrix}}\\
&=(z^1, z^2,..., z^{k+3}).
\end{align*}
\end{proof}

We next show that this definition satisfies the invariant properties we would expect a cross ratio to have. Recall that for a nonzero parameter $\lambda$, the points $(\lambda u, \lambda v)$ and $(u, v)$ in ${\bf CP}^1$ are associated with the same point in the complex plane. Which representative we choose, however, doesn't matter.

\begin{theorem}
The cross ratio is independent of the representative chosen.
\end{theorem}

\begin{proof}
Given $N+3$ nonzero parameters $\{\lambda_i\}_{i=1,...,N+3}$, we have

\begin{align*}
&(\lambda_1u^1, \lambda_2u^2, ...,\lambda_{N+2}u^{N+2}, \lambda_{N+3}u^{N+3})\\
&=\frac{[\lambda_1u^1, \lambda_3u^3,\lambda_5u^5,...,\lambda_{N+3}u^{N+3}][\lambda_2u^2,\lambda_4u^4,\lambda_5u^5,...,\lambda_{N+3}u^{N+3}]}{[\lambda_1u^1, \lambda_4u^4,\lambda_5u^5,...,\lambda_{N+3}u^{N+3}][\lambda_2u^2,\lambda_3u^3,\lambda_5u^5,...,\lambda_{N+3}u^{N+3}]}\\
&=\frac{\lambda_1\lambda_2\lambda_3\lambda_4\lambda_5^2...\lambda_{N+3}^2[u^1, u^3, u^5,..., u^{N+3}][u^2, u^4, u^5,..., u^{N+3}]}{\lambda_1\lambda_2\lambda_3\lambda_4\lambda_5^2...\lambda_{N+3}^2[u^1, u^4, u^5,..., u^{N+3}][u^2, u^3, u^5,..., u^{N+3}]}\\
&=\frac{[u^1, u^3, u^5,..., u^{N+3}][u^2, u^4, u^5,..., u^{N+3}]}{[u^1, u^4, u^5,..., u^{N+3}][u^2, u^3, u^5,..., u^{N+3}]}=(u^1, u^2, ..., u^{N+2}, u^{N+3})
\end{align*}
\end{proof}

By the same reasoning it is easy to see that the cross ratio pair is also independent of the representative chosen.

For our next result we need the following lemma.

\begin{lemma}
Given $N \times 1$ column vectors $v_1$, $v_2$,...,$v_N$,  along with an invertible $N \times N$ matrix $A$, we have

$$[Av_1, Av_2,..., Av_N]=\det(A)[v_1, v_2,..., v_N].$$

\end{lemma}

\begin{proof}
Denote the rows of $A$ by $a_i$ for $i=1,2,...,N$ and $V= \begin{pmatrix}v_1 & v_2 & \cdots & v_N \end{pmatrix}$. Then, since $\det(A)\det(V)=\det(AV)$, it suffices to show 

$$A\begin{pmatrix}v_1 & v_2 & \cdots & v_N \end{pmatrix}=\begin{pmatrix}Av_1 & Av_2 & \cdots & Av_N \end{pmatrix}.$$

We then have 

\begin{align*}
A\begin{pmatrix}v_1 & v_2 & \cdots & v_N \end{pmatrix}&=\begin{pmatrix}a_1 \\ a_2 \\ \vdots \\ a_N \end{pmatrix}\begin{pmatrix}v_1 & v_2 & \cdots & v_N \end{pmatrix}=
\begin{pmatrix}a_1 \cdot v_1 & a_1 \cdot v_2 & \cdots & a_1 \cdot v_N \\
a_2 \cdot v_1 & a_2 \cdot v_2 &  \cdots & a_2 \cdot v_N \\
 \vdots & \vdots & \ddots & \vdots \\
a_N \cdot v_1 & a_N \cdot v_2 &  \cdots & a_N \cdot v_N 
\end{pmatrix}\\
&=\begin{pmatrix}\begin{pmatrix}a_1 \\ a_2 \\ \vdots \\ a_N \end{pmatrix}v_1 & \begin{pmatrix}a_1 \\ a_2 \\ \vdots \\ a_N \end{pmatrix}v_2 & \cdots & \begin{pmatrix}a_1 \\ a_2\\ \vdots \\ a_N \end{pmatrix}v_N \end{pmatrix}=\begin{pmatrix}Av_1 & Av_2 & \cdots & Av_N \end{pmatrix}.
\end{align*}

\end{proof}

We next show that this cross ratio is invariant under associated matrices of linear fractional maps.

\begin{theorem}
Let $m_{\phi}$ be the associated matrix to an invertible linear fractional map $\phi$; that is, $\det m_{\phi} \neq 0$. Then the cross ratio is invariant under $m_{\phi}$.
\end{theorem}

\begin{proof}
Note since $[Av_1, Av_2,..., Av_N]=\det(A)[v_1, v_2,..., v_N]$, we have 

$$[m_{\phi}u^1, m_{\phi}u^2,..., m_{\phi}u^k]=\det (m_{\phi}) [u^1, u^2,..., u^k].$$

Thus 

\begin{align*}
&(m_{\phi}u^1,m_{\phi}u^2,...,m_{\phi}u^{N+3})\\
&=\frac{[m_{\phi}u^1, m_{\phi}u^3,m_{\phi}u^5,...,m_{\phi}u^{N+3}][m_{\phi}u^2,m_{\phi}u^4,m_{\phi}u^5,...,m_{\phi}u^{N+3}]}{[m_{\phi}u^1, m_{\phi}u^4,m_{\phi}u^5,...,m_{\phi}u^{N+3}][m_{\phi}u^2,m_{\phi}u^3,m_{\phi}u^5,...,m_{\phi}u^{N+3}]}\\
&=\frac{\det(m_{\phi})^2[u^1, u^3, u^5,...,u^{N+3}][u^2,u^4,u^5,...,u^{N+3}]}{\det(m_{\phi})^2[u^1, u^4,u^5,...,u^{N+3}][u^2,u^3,u^5,...,u^{N+3}]}\\
&=\frac{[u^1, u^3,u^5,...,u^{N+3}][u^2,u^4,u^5,...,u^{N+3}]}{[u^1, u^4,u^5,...,u^{N+3}][u^2,u^3,u^5,...,u^{N+3}]}=(u^1, u^2, ...,u^{N+2}, u^{N+3})
\end{align*}
where we can divide by $\det(m_{\phi})$ since we presume $m_{\phi}$ is invertible.
\end{proof}

\section{Transitivity of linear fractional maps in several complex variables.}

It is well known that linear fractional maps acting on $\mathbb{C}$ are $3$-transitive but not $4$-transitive. That is, for six points $z^1$, $z^2$, $z^3$, $w^1$, $w^2$, and $w^3$ in $\mathbb{C}$ where the $z^i$'s are pairwise distinct as are the $w^i$'s, there is a unique linear fractional map $\chi$ such that $\chi(z^i)=w^i$ for $i=1,2,3$. The usual way to achieve this result is to send the triple $(z^1, z^2, z^3)$ to the standardized points $(1,0, \infty)$ where $\infty$ represents the ``north pole" on the Riemann sphere. In homogeneous coordinates, these standardized points are associated with $\begin{pmatrix}1 \\ 1\end{pmatrix}$, $\begin{pmatrix}0 \\ 1\end{pmatrix}$ , and $\begin{pmatrix}1 \\ 0\end{pmatrix}$ , respectively. Thus, for the cross ratio as a function of $z$ with associated point $u$ given by $\phi(z)=\frac{(z-z^2)(z^1-z^3)}{(z-z^3)(z^1-z^2)}$ , we may write the point associated with $\phi (z)$ as 

$$m_{\phi} u=\begin{pmatrix}(z-z^2)(z^1-z^3) \vspace{2mm} \\   (z-z^3)(z^1-z^2) \end{pmatrix}.$$

It is clear that this sends our triple $(z^1, z^2, z^3)$ to the points $(1,0, \infty)$, respectively. Note that taking two distinct points $z^1$ and $z^2$ in $\mathbb{C}$ implies that the associated points are linearly independent in $\mathbb{C}^2$. However, it is not true that taking three distinct points in $\mathbb{C}^N$ implies the associated points will be linearly independent in $\mathbb{C}^{N+1}$. They could be collinear, for example. This potential obstruction will necessitate an additional assumption of linear independence that was not needed in $\mathbb{C}$. 

We would like our cross ratio $N$-tuple to be a generalized formula for a multidimensional linear fractional map that maps $N+2$ points to certain pre-specified points analogous to the triple $(1,0, \infty)$. To achieve this, we consider an $N+2$-tuple in which $N+1$ of the points have linearly independent associated points which are sent to the standard basis vectors $\{e_i\}_{i=1}^{N+1}$ and the remaining associated point is sent to $\sum_{i=1}^{N+1}e_i$. To begin, we see how to do this in $\mathbb{C}^2$ and generalize. 

Let $u^i \in {\bf CP}^2$ be associated with $z^i \in \mathbb{C}^2$ for $i=2,...,5$ and $u$ associated with $z$. Define the linear fractional map $\phi: \mathbb{C}^2 \rightarrow \mathbb{C}^2$ by

\begin{align*}
\phi(z)&=(z, z^2 , z^3, z^4,z^5)_2=(u, u^2, u^3, u^4, u^5)_2\\
&=\left( \frac{[u, u^3,u^5][u^2,u^4,u^5][u^2,u^3,u^4]}{[u, u^4,u^5][u^2,u^3,u^4][u^2,u^3,u^5]}, \frac{[u,u^3,u^4][u^2,u^3,u^5][u^2,u^4,u^5]}{[u, u^4,u^5][u^2,u^3,u^4][u^2,u^3,u^5]} \right).
\end{align*}

If the point $u \in {\bf CP}^2$ is associated with $z \in \mathbb{C}^2$, then the point $v \in {\bf CP}^2$ associated with $w=\phi(z) \in \mathbb{C}^2$ is given by

\begin{equation}\label{cr}
v=m_{\phi}u=\begin{pmatrix}
        [u, u^3,u^5][u^2,u^4,u^5][u^2,u^3,u^4] \\
        [u,u^3,u^4][u^2,u^3,u^5][u^2,u^4,u^5]\\
        [u, u^4,u^5][u^2,u^3,u^4][u^2,u^3,u^5]
     \end{pmatrix}.
\end{equation}

We recall that $[u^i, u^j, u^k]=0$ when any of $i,j,k$ are equal or, more generally, when there is linear dependence. To keep the cross ratio map given above well-defined, we introduce the following hypothesis.

\begin{definition}
Let the set $\{u^i\}_{i=3}^5$ in ${\bf CP}^2$ form a linearly independent set. We make the assumption that $u^2=\alpha u^3+\beta u^4+\gamma u^5$ where $\alpha$, $\beta$, and $\gamma$ are all non-zero. We will call this {\it the independence hypothesis}. 
\end{definition}

Considering the three-dimensional space $(\alpha, \beta, \gamma)$, the independence hypothesis asks only that we omit a set of measure zero. We generalize this hypothesis to higher dimensions in the natural way.

With this in mind, it is clear to see that we have $m_{\phi}u^2=(1,1,1)$, $m_{\phi}u^3=(0,0,1)$, $m_{\phi}u^4=(1,0,0)$, and $m_{\phi}u^5=(0,1,0)$ as desired where $(1,1,1)$ is associated with $(1,1)$, $(0,0,1)$ with $(0,0)$, and vectors such as $(1,0,0)$, which we will denote by $e_{1, \infty}$, and $(0,1,0)$, which we will denote by $e_{2, \infty}$, correspond to the points in the hyperplane at infinity that are tangent to $(0,1)$ and  $(1, 0)$, respectively. In general, our goal will then be to map $N+2$ points associated with $N+2$ vectors in ${\bf CP}^N$ satisfying the independence hypothesis to a set of $N+2$ standardized points in $\mathbb{C}^N$ with associated points $\sum_{k=1}^{N+1}e_k$ and $\{e_k\}_{k=1}^{N+1}$ where $e_i$ is the vector that is zero on every coordinate except the $i$th coordinate where it takes the value of $1$. 

\begin{theorem}
Given $N+2$ distinct elements $z^2$, $z^3$,...,$z^{N+3}$ in $\mathbb{C}^N$ with respective associated points $u^2$, $u^3$,...,$u^{N+3}$ in ${\bf CP}^N$ that satisfy the independence hypothesis, there exists a unique linear fractional map $\phi$ such that $m_{\phi}u^2=(1,1,...,1)=\sum_{k=1}^{N+1}e_k$, $m_{\phi}u^3=e_{N+1}$, $m_{\phi}u^4=e_1$, $m_{\phi}u^5=e_2$, and $m_{\phi}u^i=e_{i-3}$ for $i>5$.
\end{theorem}

\begin{proof}
Define the linear fractional map $\phi$ by

\begin{align*}\phi(z)&=(z^1, z^2,..., z^{N+2}, z^{N+3})_N \\
&=(u^1, u^2,..., u^{N+2}, u^{N+3})_N=\left( \bigg\{\frac{[u^2, u^i]_N^c[u, u^3]_N^c}{[u, u^i]_N^c[u^2, u^3]_N^c}\bigg\}_{i=4,...,N+3}\right).
\end{align*}

We may then write the point associated with $\phi(z)$ as

$$v=m_{\phi}u=\begin{pmatrix}\frac{[u^2, u^4]_N^c}{[u, u^4]_N^c} & \frac{[u^2, u^5]_N^c}{[u, u^5]_N^c}  & \cdots &  \frac{[u^2, u^3]_N^c}{[u^,u^3]_N^c} \end{pmatrix}^T.$$

We recall that $[u^i, u^j]_N^c=[u,...,u^{i-1},u^{i+1},...,u^{j-1},u^{j+1},...u^{N+3}]=0$ when $u=u^k$ for some integer $1<k<N+3$, $k \neq i,j$. In particular, note that none of the denominators in the entries of $v$ are zero. Thus we have $m_{\phi}u^2=\sum_{k=1}^{N+1}e_k$, $m_{\phi}u^3=e_{N+1}$, $m_{\phi}u^4=e_1$, $m_{\phi}u^5=e_2$, and $m_{\phi}u^k=e_{k-3}$. 

We next demonstrate uniqueness. Suppose there are two linear fractional maps $\phi$ and $\psi$ such that the associated matrices $m_{\phi}$ and $m_{\psi}$ send each $u^i$ as above. Then the associated matrix $M=m_{\phi}m_{\psi}^{-1}$ fixes $e_i$ and $p=\sum_{i=1}^Ne_i$. Denote $M$ by

$$ M=\begin{pmatrix}
    a_{11} & a_{12} & a_{13} & \dots  & a_{1N} \\
    a_{21} & a_{22} & a_{23} & \dots  & a_{2N} \\
    \vdots & \vdots & \vdots & \ddots & \vdots \\
    a_{N1} & a_{N2} & a_{N3} & \dots  & a_{NN}
\end{pmatrix}.
$$

Then $Me_i=e_i$ implies $a_{ij}=0$ for $i \neq j$ and $Mp=p$ implies $a_{ii}=a_{jj}$ for all $1 \leq i,j \leq N$. We conclude that $M=a_{11}I$ where $I$ is the identity which implies $m_{\phi}=a_{11}m_{\psi}$ and thus $\phi=\psi$.
\end{proof}

To generalize in several variables, we use transitivity in a more restricted sense. In particular, we make the additional assumption that the initial data points in $\mathbb{C}^N$ have associated points in ${\bf CP}^N$ that satisfy the independence hypothesis. We next show that, in this sense of transitivity, linear fractional maps in $\mathbb{C}^N$ are $(N+2)$-transitive but not $(N+3)$-transitive. We began with some theorems.

\begin{theorem}
Given $N+2$ pairwise distinct points $\{z^i\}_{i=2}^{N+3}$ in $\mathbb{C}^N$ with associated points $\{u^i\}_{i=2}^{N+3}$ and $N+2$ pairwise distinct points $\{w^i\}_{i=2}^{N+3}$ in $\mathbb{C}^N$ with associated points $\{v^i\}_{i=2}^{N+3}$ such that the sets $\{u^i\}_{i=2}^{N+3}$ and $\{v^i\}_{i=2}^{N+3}$ satisfy the independence hypothesis, there exists a unique linear fractional map $\chi(z)$ such that $\chi(z^i)=w^i$ for $i=2,...,N+3$.
\end{theorem}

\begin{proof}
By Theorem $4$, there exist maps $\phi$ and $\psi$ such that $m_{\phi}u^2=m_{\psi}v^2=\sum_{k=1}^{N+1}e_k$, $m_{\phi}u^3=m_{\psi}v^3=e_{N+1}$, $m_{\phi}u^4=m_{\psi}v^4=e_1$, $m_{\phi}u^5=m_{\psi}v^5=e_2$, and $m_{\phi}u^k=m_{\psi}v^k=e_{k-3}$ otherwise. Define $m_{\gamma}=m_{\psi}^{-1}m_{\phi}$. This sends $u^i$ to $v^i$ for $i=2,...,N+3$. To show uniqueness, suppose $m_{\sigma}$ maps $u^i$ to $v^i$ for $i=2,...,N+3$. Then $m_{\phi}$ and $m_{\psi}m_{\sigma}$ each send $u^i$ to the same values for $i=2,...,N+3$. Thus by Theorem $4$, we have $m_{\phi}=m_{\psi}m_{\sigma}$ which implies $m_{\gamma}=m_{\psi}^{-1}m_{\phi}=m_{\sigma}$ from which we conclude that $\gamma=\sigma$ as desired.
\end{proof}

\begin{theorem}
If $m_{\phi}$ fixes $N+2$ distinct points $\{u^i\}_1^{N+2}$ satisfying the independence hypothesis, then $\phi$ is the identity.
\end{theorem}

\begin{proof}
Suppose that $m_{\phi}$ fixes $N+2$ distinct points $\{u^i\}_1^{N+2}$ satisfying the independence hypothesis, then, since the identity also fixes these points, by Theorem $5$ we have that $m_{\phi}$ is equal to the identity and thus $\phi$ is the identity.
\end{proof}

We can thus conclude that our linear fractional maps are $(N+2)$-transitive but they are not $(N+3)$-transitive.

\section{More to explore.}

The author invites the reader to investigate  this generalized definition further. For example, in one variable, we have a well-known criteria to determine when the cross ratio is real.
\begin{theorem}
The cross ratio $(z^1,z^2,z^3,z^4)$  is real if and only if the four points lie on a circle (where we let a line be a circle of infinite radius). 
\end{theorem}

\begin{proof}
Consider the cross ratio as a function of the first argument, $\phi(z)=(z,z^1,z^2,z^3)$. We saw that $\phi$ maps the triple $(z^1,z^2,z^3)$ to $(1,0,\infty)$, respectively. Since linear fractional maps send circles to circles, we conclude that $\phi$ sends our circle containing the triple $(z^1, z^2, z^3)$ to the extended real line. Now consider the map $\phi^{-1} \left(\phi (z)\right)=z$. It is known that linear fractional maps are bijective and thus $\phi(z)$ is real if and only if $z=\phi^{-1}\left( \phi(z\right))$ lies on the circle containing our triple $(z^1, z^2, z^3)$ which is the image of the extended real line under our map $\phi^{-1}$. 
\end{proof}

In order for this theorem to extend to our generalized cross ratio $N$-tuple, it will take an upgrade of some geometric facts about these generalized maps. We appeal to results by Cowen and MacCluer (see \cite{Cowen}, Section 3 for more details.). In particular, one finds more flexibility when working in higher dimensions. The appropriate generalization of circles in this situation is to consider {\it ellipsoids}, where an ellipsoid is a translate of the image of the unit ball under an invertible complex linear transformation. We recall Theorem $6$ from [1].

\begin{theorem}
If $\phi$ is a one-to-one linear fractional map defined on a ball $\overline{\mathbb{B}}$ in $\mathbb{C}^N$, then $\phi (\mathbb{B})$ is an ellipsoid.
\end{theorem}

\begin{proof}
See [1], section $3$.
\end{proof}

This brings us to the following open question, left for the reader to demonstrate.

\begin{question}
Is there an analogue to Theorem $7$ for the generalized cross ratio if we relax the presumption of points lying on the boundary of a sphere and instead use the boundary of an ellipsoid?
\end{question}

In some form or another, this quantity and its special invariant properties have been examined and celebrated since antiquity. Now that we have a new perspective to talk about this quantity in higher dimensions, the author challenges the reader to explore what other similarities or differences the generalized version of this classical quantity has.

\end{document}